\documentclass[a4paper]{article}%
\usepackage{amsmath}
\usepackage{amsfonts}
\usepackage{amssymb}
\usepackage{graphicx}%
\newtheorem{theorem}{Theorem}

\newtheorem{corollary}[theorem]{Corollary}

\newtheorem{lemma}[theorem]{Lemma}

\newtheorem{proposition}[theorem]{Proposition}
\newtheorem{remark}[theorem]{Remark}

\newenvironment{proof}[1][Proof]{\noindent\textbf{#1.} }{\ \rule{0.5em}{0.5em}}
\begin{document}

\title{Dirichlet and Neumann problems\\for Klein-Gordon-Maxwell systems \thanks{The authors are supported by M.I.U.R.
- P.R.I.N. ``Metodi variazionali e topologici nello studio di fenomeni non
lineari''.}}
\author{Pietro d'Avenia \thanks{Dipartimento di Matematica, Politecnico di Bari, Via
Orabona, 4, 70125 Bari (Italy), \texttt{p.davenia@poliba.it}}
\and Lorenzo Pisani \thanks{Dipartimento di Matematica, Universit\`a degli Studi di
Bari, Via Orabona, 4, 70125 Bari (Italy), \texttt{pisani@dm.uniba.it,
siciliano@dm.uniba.it}}
\and Gaetano Siciliano \addtocounter{footnote}{-1} \footnotemark}
\date{}
\maketitle

\begin{abstract}
This paper deals with the Klein-Gordon-Maxwell system in a bounded spatial
domain. We study the existence of solutions having a specific form, namely
standing waves in equilibrium with a purely electrostatic field. We prescribe
Dirichlet boundary conditions on the matter field, and either Dirichlet or
Neumann boundary conditions on the electric potential.

MSC2000: 35J55, 35J65, 35J50, 35Q40, 35Q60.

\end{abstract}

\section{Introduction}

In this paper we pursue the investigation of the existence and multiplicity of
solutions for a class of Klein-Gordon-Maxwell (KGM for short) systems in a
bounded spatial domain.

Let us recall the general setting for KGM systems. We are concerned with a
matter field $\psi$, whose free Lagrangian density is given by%
\begin{equation}
\mathcal{L}_{KG}=\frac{1}{2}\left(  \left\vert \partial_{t}\psi\right\vert
^{2}-\left\vert \nabla\psi\right\vert ^{2}-m^{2}\left\vert \psi\right\vert
^{2}\right)  , \label{l0}%
\end{equation}
with $m>0$. The field is charged and in equilibrium with its own
electromagnetic field $(\mathbf{E},\mathbf{B})$, represented by means of the
gauge potentials $(\phi,\mathbf{A})$,%
\[
\mathbf{E}=-\left(  \nabla\phi+{\partial_{t}\mathbf{A}}\right)  ,\qquad
\mathbf{B}=\nabla\times\mathbf{A}.
\]
Abelian gauge theories provide a model for the interaction: formally we
replace the ordinary derivatives $\left(  \partial_{t},\nabla\right)  $ in
(\ref{l0}) with the so-called gauge covariant derivatives
\[
\left(  \partial_{t}+iq\phi,\nabla-iq\mathbf{A}\right)  ,
\]
where $q$ is a \emph{nonzero} coupling constant (see \emph{e.g.} \cite{fel}).
Since the electromagnetic field is not prescribed, the total Lagrangian
density contains also the term%
\[
\mathcal{L}_{M}=\frac{1}{8\pi}\left(  \left\vert \mathbf{E}\right\vert
^{2}-\left\vert \mathbf{B}\right\vert ^{2}\right)  .
\]
The KGM system is given by the Euler-Lagrange equations corresponding to the
Lagrangian density
\begin{align*}
\mathcal{L}_{KGM}  &  =\mathcal{L}_{KG}(\psi,\phi,\mathbf{A})+\mathcal{L}%
_{M}(\phi,\mathbf{A})=\\
&  =\frac{1}{2}\left(  \left\vert \left(  \partial_{t}+iq\phi\right)
\psi\right\vert ^{2}-\left\vert \left(  \nabla-iq\mathbf{A}\right)
\psi\right\vert ^{2}-m^{2}\left\vert \psi\right\vert ^{2}\right)  +\\
&  +\frac{1}{8\pi}\left(  \left\vert \nabla\phi+{\partial_{t}\mathbf{A}%
}\right\vert ^{2}-\left\vert \nabla\times\mathbf{A}\right\vert ^{2}\right)  .
\end{align*}

The study of KGM systems is usually carried out for special classes of
solutions (and for some classes of lower-order nonlinear perturbations in
$\mathcal{L}_{KG}$). Here we consider%
\[
\psi=u(x)e^{-i\omega t},\qquad\phi=\phi\left(  x\right)  ,\qquad
\mathbf{A}=\mathbf{0},
\]
that is, a standing wave in equilibrium with a purely electrostatic field
\[
\mathbf{E}=-\nabla\phi\left(  x\right)  ,\qquad\mathbf{B}=\mathbf{0}.
\]
Under this ansatz, the KGM system reduces to
\begin{equation}
\left\{
\begin{array}
[c]{l}%
-\Delta u-\left(  q\phi-\omega\right)  ^{2}u+m^{2}u=0,\\
\Delta\phi=4\pi q\left(  q\phi-\omega\right)  u^{2}%
\end{array}
\right.  \tag{S0}\label{Main0}%
\end{equation}
(see \cite{bf2} where the set of equations in the general case has been derived).

We will study (\ref{Main0}) in a bounded domain $\Omega\subset\mathbf{R}^{3}$
with smooth boundary $\partial\Omega$. The unknowns are the real functions $u$
and $\phi$ defined on $\Omega$ and the frequency $\omega\in\mathbf{R}$. We are
interested in finding \emph{nontrivial} solutions, that is, solutions such
that $u\neq0$. We shall consider two different boundary conditions, specifically,

\begin{itemize}
\item either Dirichlet boundary conditions%
\begin{equation}
\left\{
\begin{array}
[c]{l}%
u=h\\
\phi=\zeta
\end{array}
\right.  \quad\text{on }\partial\Omega, \tag{D0}\label{Dir0}%
\end{equation}

\item or \textquotedblleft mixed\textquotedblright\ boundary conditions%
\begin{equation}
\left\{
\begin{array}
[c]{l}%
u=h\\
\frac{\partial\phi}{\partial\mathbf{n}}=\theta
\end{array}
\right.  \quad\text{on }\partial\Omega, \tag{M0}\label{Mix0}%
\end{equation}
that is, Dirichlet boundary conditions on $u$ and Neumann boundary conditions
on $\phi$,
\end{itemize}

\noindent where $h$, $\zeta$ and $\theta$ are smooth functions defined on the
boundary $\partial\Omega$.

With $q\neq0$, the change of variables%
\begin{equation}
\left\{
\begin{array}
[c]{l}%
u_{q}=\sqrt{4\pi}q\,u,\\
\phi_{q}=q\phi-\omega
\end{array}
\right.  \label{ch var}%
\end{equation}
transforms the system (\ref{Main0}) and the boundary conditions (\ref{Dir0})
and (\ref{Mix0}) into%
\begin{equation}
\left\{
\begin{array}
[c]{l}%
-\Delta u_{q}-\phi_{q}^{2}u_{q}+m^{2}u_{q}=0,\\
\Delta\phi_{q}=\phi_{q}u_{q}^{2},
\end{array}
\right.  \tag{S1}\label{Main1}%
\end{equation}%
\begin{equation}
\left\{
\begin{array}
[c]{l}%
u_{q}=\sqrt{4\pi}q\,h\\
\phi_{q}=q\,\zeta-\omega
\end{array}
\right.  \quad\text{on }\partial\Omega,\tag{D1}\label{Dir1}%
\end{equation}%
\begin{equation}
\left\{
\begin{array}
[c]{l}%
u_{q}=\sqrt{4\pi}q\,h\\
\frac{\partial\phi_{q}}{\partial\mathbf{n}}=q\,\theta
\end{array}
\right.  \quad\text{on }\partial\Omega,\tag{M1}\label{Mix1}%
\end{equation}
respectively.

First we study problem (\ref{Main1})-(\ref{Dir1}). Let $\left\{  \lambda
_{k}\right\}  $ denote the eigenvalues of $-\Delta$ with homogeneous Dirichlet
boundary conditions.

\begin{theorem}
\label{th1}Assume%
\begin{equation}
\left\Vert q\zeta-\omega\right\Vert _{\infty}^{2}<m^{2}+\lambda_{1}.
\label{hyp}%
\end{equation}

\begin{enumerate}
\item If $h\neq0$, the problem (\ref{Main1})-(\ref{Dir1}) has a nontrivial solution.

\item If $h=0$, the problem (\ref{Main1})-(\ref{Dir1}) has no nontrivial solutions.
\end{enumerate}
\end{theorem}

It is immediately seen that (\ref{Main1})-(\ref{Dir1}) has a trivial solution
if and only if $h=0$. Hence from Theorem \ref{th1} we deduce the following

\begin{corollary}
Assume (\ref{hyp}). Then problem (\ref{Main1})-(\ref{Dir1}) has a solution.
This solution is trivial if and only if $h=0$.
\end{corollary}

\begin{remark}
If we assume%
\begin{equation}
\omega^{2}<m^{2}+\lambda_{1}\text{,} \label{hyp0}%
\end{equation}
then (\ref{hyp}) is satisfied whenever $\left\vert q\right\vert $ is
sufficiently small. Then it is interesting to study the limit case $q=0$.

Since the change of variables (\ref{ch var}) is not allowed for $q=0$, we
consider the \textquotedblleft original\textquotedblright\ problem
(\ref{Main0})-(\ref{Dir0}). Being uncoupled, it can be splitted into%
\begin{equation}
\left\{
\begin{array}
[c]{ll}%
-\Delta u-\left(  \omega^{2}-m^{2}\right)  u=0,\quad & \\
u=h & \text{on }\partial\Omega
\end{array}
\right.  \label{split u}%
\end{equation}
and%
\begin{equation}
\left\{
\begin{array}
[c]{ll}%
\Delta\phi=0,\quad & \\
\phi=\zeta & \text{on }\partial\Omega.
\end{array}
\right.  \label{split phi}%
\end{equation}
Problem (\ref{split phi}) has a unique solution (independent of $u$). The
existence and uniqueness of solutions of problem (\ref{split u}) depend on the
value $\omega^{2}-m^{2}$. If $\omega^{2}-m^{2}\neq\lambda_{k}$ (in particular
if (\ref{hyp0}) holds), then problem (\ref{split u}) has a unique solution;
this solution is nontrivial if and only if $h\neq0$. Hence, at least in the
case (\ref{hyp0}), the existence of a (nontrivial) solution depends
continuously\ on $q$, as $q\rightarrow0$.
\end{remark}

Now we address problem (\ref{Main1})-(\ref{Mix1}).

First we notice that, after the change of variables (\ref{ch var}), which is
valid for $q\neq0$, the system does not depend on the frequency $\omega$.
Hence, for any $\omega\in\mathbf{R}$, the existence of a standing wave
$\psi=u(x)e^{-i\omega t}$ in equilibrium with a purely electrostatic field is
equivalent to the existence of a static matter field $\psi=u\left(  x\right)
$, in equilibrium with the same electric field.

\begin{theorem}
\label{thmix}If $\left\vert q\right\vert \!\left\Vert \theta\right\Vert
_{H^{1/2}\left(  \partial\Omega\right)  }$ is sufficiently small and $h\neq0$,
then there exists a nontrivial solution of (\ref{Main1})-(\ref{Mix1}).
\end{theorem}

It is easily seen that the problem (\ref{Main1})-(\ref{Mix1}) has infinitely
many trivial solutions when $h=0$ and $\int_{\partial\Omega}\theta\,d\sigma=0$.

The case $h\neq0$ is covered by Theorem \ref{thmix}.

In \cite[Theorem 1.1]{DPS2007} we have shown that, if $\left\vert q\right\vert
\!\left\Vert \theta\right\Vert _{H^{1/2}\left(  \partial\Omega\right)  }$ is
sufficiently small, $h=0$ and $\int_{\partial\Omega}\theta\,d\sigma\neq0$,
then problem (\ref{Main1})-(\ref{Mix1}) has a nontrivial solution.

Taking into account the Neumann boundary condition on $\phi_{q}$, the second
equation of (\ref{Main1}) gives the following necessary condition%
\[
\int_{\Omega}\phi_{q}u_{q}^{2}\,dx=q\int_{\partial\Omega}\theta\,d\sigma.
\]
Hence, if (\ref{Main1})-(\ref{Mix1}) has trivial solutions (\emph{i.e.}
$u_{q}=0$), then we have necessarily $\int_{\partial\Omega}\theta\,d\sigma=0$
and, of course, $h=0$. Vice versa, in the same paper \cite[Theorem
1.1]{DPS2007} we have shown that, if $\left\vert q\right\vert \!\left\Vert
\theta\right\Vert _{H^{1/2}\left(  \partial\Omega\right)  }$ is sufficiently
small, the joint conditions $h=0$ and $\int_{\partial\Omega}\theta\,d\sigma=0$
force $u_{q}$ to be $0$.

All these results are summarized in the following statement.

\begin{theorem}
\label{thmix1}If $\left\vert q\right\vert \!\left\Vert \theta\right\Vert
_{H^{1/2}\left(  \partial\Omega\right)  }$ is sufficiently small, then problem
(\ref{Main1})-(\ref{Mix1}) has a solution. This solution is trivial if and
only if $h=0$ and $\int_{\partial\Omega}\theta\,d\sigma=0$.
\end{theorem}

\begin{remark}
Under the boundary conditions (\ref{Mix0}), the existence of solutions
(trivial or nontrivial) of (\ref{Main0}) is not continuous with respect to
$q\rightarrow0$ in the following sense. If we fix boundary data such that
$\int_{\partial\Omega}\theta\,d\sigma\neq0$, Theorem \ref{thmix1} gives (via
(\ref{ch var})) a nontrivial solution of (\ref{Main0})-(\ref{Mix0}) for all
$q\neq0$ sufficiently small. However the limit problem\ has no solutions.
Indeed, for $q=0$, (\ref{Main0}) decouples into (\ref{split u}) and
\[
\left\{
\begin{array}
[c]{ll}%
\Delta\phi=0, & \\
\frac{\partial\phi}{\partial\mathbf{n}}=\theta\quad & \text{on }\partial
\Omega,
\end{array}
\right.
\]
and the latter system has a solution if and only if $\int_{\partial\Omega
}\theta\,d\sigma=0$. Moreover, unlike the case $q\neq0$, the limit problem
depends on $\omega$.
\end{remark}

From now on, for the sake of simplicity, we shall omit the subscript $q$.

As we said before, we can consider a nonlinear lower order term $g(x,u)$ in
the first equation of (\ref{Main1}). Thus we obtain the system%
\begin{equation}
\left\{
\begin{array}
[c]{l}%
-\Delta u-\phi^{2}u+m^{2}u-g\left(  x,u\right)  =0,\\
\Delta\phi=\phi u^{2}.
\end{array}
\right.  \tag{S2}\label{prob1nl}%
\end{equation}

We assume that $g$ behaves like $\left\vert u\right\vert ^{p-2}u$ with
$p\in(2,6)$. More precisely, $g\in C\left(  \bar{\Omega}\times\mathbf{R}%
,\mathbf{R}\right)  $ satisfies the following well-known Ambrosetti-Rabinowitz
conditions (see \emph{e.g.} \cite{rab}):

\begin{description}
\item[($\mathbf{g}_{1}$)] there exist $a_{1},a_{2}\geq0$ and $\,p\in\left(
2,6\right)  $ such that
\[
\left\vert g\left(  x,t\right)  \right\vert \leq a_{1}+a_{2}\left\vert
t\right\vert ^{p-1};
\]

\item[($\mathbf{g}_{2}$)] $g\left(  x,t\right)  =o\left(  \left\vert
t\right\vert \right)  $ as $t\rightarrow0$ uniformly in $x$;

\item[($\mathbf{g}_{3}$)] there exist $s\in\left(  2,p\right]  $ and $r\geq0$
such that
\[
0<sG\left(  x,t\right)  \leq tg\left(  x,t\right)  ,
\]
for every $\left\vert t\right\vert \geq r$, where
\[
G\left(  x,t\right)  =\int_{0}^{t}g\left(  x,\tau\right)  \,d\tau.
\]

\end{description}

\begin{theorem}
\label{thnonlin}Let us consider the system (\ref{prob1nl}) with boundary
conditions%
\begin{equation}
\left\{
\begin{array}
[c]{l}%
u=0\\
\phi=q\zeta-\omega
\end{array}
\quad\text{on }\partial\Omega.\right.  \tag{D2}\label{Dir2}%
\end{equation}

\begin{enumerate}
\item If (\ref{hyp}) holds, then there exists a nontrivial solution.

\item If $g$ is odd, for every $m,\omega,q$, there exist infinitely many
solutions $\left(  u_{i},\phi_{i}\right)  \in$ $H_{0}^{1}\left(
\Omega\right)  \times H^{1}\left(  \Omega\right)  $ with $\left\Vert \nabla
u_{i}\right\Vert _{2}\rightarrow+\infty$ and $\left\{  \phi_{i}\right\}  $
bounded in $L^{\infty}\left(  \Omega\right)  $.
\end{enumerate}
\end{theorem}

The system (\ref{prob1nl}) with mixed boundary conditions has been studied in
\cite[Theorem 1.3]{DPS2007}. For the sake of completeness, we quote the statement.

\begin{theorem}
Let us consider the system (\ref{prob1nl}) with boundary conditions%
\[
\left\{
\begin{array}
[c]{l}%
u=0\\
\dfrac{\partial\phi}{\partial\mathbf{n}}=q\,\theta
\end{array}
\right.  \quad\text{on }\partial\Omega.
\]
Assume that $\int_{\partial\Omega}\theta\,d\sigma=0$.

\begin{enumerate}
\item If $q\!\left\Vert \theta\right\Vert _{H^{1/2}\left(  \partial
\Omega\right)  }$ is sufficiently small, then there exists a nontrivial solution.

\item If $g$ is odd, for every $m,\omega,q$, there exist infinitely many
solutions $\left(  u_{i},\phi_{i}\right)  \in$ $H_{0}^{1}\left(
\Omega\right)  \times H^{1}\left(  \Omega\right)  $, with $\left\Vert \nabla
u_{i}\right\Vert _{2}\rightarrow+\infty$ and $\left\{  \phi_{i}\right\}  $
bounded in $L^{\infty}\left(  \Omega\right)  $.
\end{enumerate}
\end{theorem}

We conclude this section by recalling some results which have motivated our
research. The application of global variational methods to the study of KGM
systems started with the pioneering paper of Benci and Fortunato \cite{bf2}.
They studied the nonlinear KGM system
\begin{equation}
\left\{
\begin{array}
[c]{l}%
-\Delta u-\left(  q\phi-\omega\right)  ^{2}u+m^{2}u=\left\vert u\right\vert
^{p-2}u,\\
\Delta\phi=4\pi q\left(  q\phi-\omega\right)  u^{2}%
\end{array}
\right.  \label{Main3}%
\end{equation}
in the whole space $\mathbf{R}^{3}$. They proved the existence of infinitely
many solutions if $p\in(4,6)$ and $\omega^{2}<m^{2}$. Then D'Aprile and Mugnai
(\cite{dm2}) proved two interesting nonexistence results: the
\textquotedblleft linear\textquotedblright\ system (\ref{Main0}) has no
nontrivial finite energy solution in $\mathbf{R}^{3}$; the \textquotedblleft
nonlinear\textquotedblright\ system (\ref{Main3}) has no nontrivial finite
energy solutions in $\mathbf{R}^{3}$ if $p\notin\lbrack2,6]$. The existence
result in \cite{bf2} has been generalized in \cite{dm1} and then in
\cite{bf4}. In \cite{bf4}, Benci and Fortunato prove the existence of
solutions of (\ref{Main3}) when $p\in(2,6)$ and
\[
\left\vert \omega\right\vert <m\,\sqrt{\min\left(  1,\frac{p-2}{2}\right)  }.
\]
Further results on this topic are contained in \cite{Cass}, \cite{D'Av-Pis}.

On the other hand, the lower-order term $\left\vert u\right\vert ^{p-2}u$ in
(\ref{Main3}) is not suitable for physical models since
\[
W(u)=\frac{1}{2}m^{2}u^{2}-\frac{1}{p}\left\vert u\right\vert ^{p}%
\]
is not positive and the conservation of the energy does not guarantee global
existence for the initial value problems. Some recent papers (\cite{bf5},
\cite{Long}) are concerned with systems%
\begin{equation}
\left\{
\begin{array}
[c]{l}%
-\Delta u-\left(  q\phi-\omega\right)  ^{2}u+m^{2}u=G^{\prime}(u),\\
\Delta\phi=4\pi q\left(  q\phi-\omega\right)  u^{2},
\end{array}
\right.  \label{Main4}%
\end{equation}
where%
\[
W(u)=\frac{1}{2}m^{2}u^{2}-G(u)\geq0.
\]
In \cite{bf5} it is shown that there exist solutions of (\ref{Main4}) if the
coupling constant $q$ is sufficiently small. It is easy to see that our
Theorems \ref{th1} and \ref{thmix} (concerning the system (\ref{Main0})) are
consistent with this kind of results.

Lastly we recall that Benci and Fortunato have introduced a different class of
solutions of the KGM system: three-dimensional vortices, \emph{i.e.} solutions
such that the matter field $\psi$ has nontrivial angular momentum and the
corresponding electromagnetic field has nontrivial magnetic component (see
\cite{fort} and references therein).

\section{\label{DDsect}Proof of Theorem \ref{th1}}

To get homogeneous boundary conditions, we change variables as follows%
\[
v=u-U,\qquad\varphi=\phi-\Phi_{D},
\]
where $U$ and $\Phi_{D}$ are the solutions of%
\begin{align}
&  \left\{
\begin{array}
[c]{ll}%
-\Delta U+m^{2}U=0,\quad & \\
U=\sqrt{4\pi}qh & \text{on }\partial\Omega,
\end{array}
\right. \label{def U}\\
&  \left\{
\begin{array}
[c]{ll}%
\Delta\Phi_{D}=0, & \\
\Phi_{D}=q\zeta-\omega\quad & \text{on }\partial\Omega.
\end{array}
\right. \nonumber
\end{align}
By (\ref{hyp}) we have%
\begin{equation}
\left\Vert \Phi_{D}\right\Vert _{\infty}^{2}=\left\Vert q\zeta-\omega
\right\Vert _{\infty}^{2}<m^{2}+\lambda_{1}. \label{hyp2}%
\end{equation}

Then problem (\ref{Main1})-(\ref{Dir1}) can be written as%
\begin{equation}
\left\{
\begin{array}
[c]{ll}%
-\Delta v-\left(  \varphi+\Phi_{D}\right)  ^{2}\left(  v+U\right)
+m^{2}v=0,\quad & \\
\Delta\varphi=\left(  \varphi+\Phi_{D}\right)  \left(  v+U\right)  ^{2}, & \\
v=\varphi=0 & \text{on }\partial\Omega.
\end{array}
\right.  \label{prob2}%
\end{equation}
The solutions of (\ref{prob2}) are the critical points of the functional%
\[
F\left(  v,\varphi\right)  =\frac{1}{2}\left\Vert \nabla v\right\Vert _{2}%
^{2}-\frac{1}{2}\int_{\Omega}\left(  \varphi+\Phi_{D}\right)  ^{2}\left(
v+U\right)  ^{2}dx+\frac{m^{2}}{2}\left\Vert v\right\Vert _{2}^{2}-\frac{1}%
{2}\left\Vert \nabla\varphi\right\Vert _{2}^{2},
\]
defined in $H_{0}^{1}\left(  \Omega\right)  \times H_{0}^{1}\left(
\Omega\right)  $. The functional $F$ is strongly indefinite, so we apply a
well known reduction argument (see \emph{e.g.} \cite{bf2}).

Let $\varphi_{v}\in H_{0}^{1}\left(  \Omega\right)  $ denote the unique
solution of%
\[
\left\{
\begin{array}
[c]{ll}%
\Delta\varphi=\left(  \varphi+\Phi_{D}\right)  \left(  v+U\right)  ^{2},\quad
& \\
\varphi=0 & \text{on }\partial\Omega.
\end{array}
\right.
\]
Since the map $v\mapsto\varphi_{v}$ is $C^{1}$, we define the \emph{reduced}
$C^{1}$ functional $J\left(  v\right)  =F\left(  v,\varphi_{v}\right)  $. It
is easy to see that the pair $(v,\varphi)$ is a solution of (\ref{prob2}) if
and only if $v$ is a critical point of $J$ and $\varphi=\varphi_{v}$. So we
have to prove the existence of a critical point of $J$.

The function $\varphi_{v}$ satisfies
\begin{equation}
\left\Vert \nabla\varphi_{v}\right\Vert _{2}^{2}+\int_{\Omega}\varphi_{v}%
^{2}\left(  v+U\right)  ^{2}dx=-\int_{\Omega}\varphi_{v}\Phi_{D}\left(
v+U\right)  ^{2}dx. \label{varphiquad}%
\end{equation}
Then we obtain%

\[
J\left(  v\right)  =\frac{1}{2}\left\Vert \nabla v\right\Vert _{2}^{2}%
+\frac{m^{2}}{2}\left\Vert v\right\Vert _{2}^{2}-\frac{1}{2}\int_{\Omega}%
\Phi_{D}\left(  \varphi_{v}+\Phi_{D}\right)  \left(  v+U\right)  ^{2}dx
\]
and, for every $w\in H_{0}^{1}\left(  \Omega\right)  $,%
\[
\left\langle J^{\prime}\left(  v\right)  ,w\right\rangle =\int_{\Omega}\nabla
v\nabla w\,dx+m^{2}\int_{\Omega}vw\,dx-\int_{\Omega}\left(  \varphi_{v}%
+\Phi_{D}\right)  ^{2}\left(  v+U\right)  w\,dx.
\]
The following lemma shows that the functions $\left\{  \varphi_{v}\right\}  $
are uniformly bounded in $L^{\infty}\left(  \Omega\right)  $.

\begin{lemma}
\label{stimaphi}For every $v\in H_{0}^{1}\left(  \Omega\right)  $, the
function $\varphi_{v}$ satisfies the following inequalities%
\begin{equation}
-\max\left\{  0,\Phi_{D}\right\}  =-\Phi_{D}^{+}\leq\varphi_{v}\leq\Phi
_{D}^{-}=\max\left\{  0,-\Phi_{D}\right\}  ,\qquad\text{a.e. in }\Omega.
\label{stimamaxmin}%
\end{equation}
Hence $\varphi_{v}\in L^{\infty}(\Omega)$ and satisfies%
\begin{align}
\left\Vert \varphi_{v}\right\Vert _{\infty}  &  \leq\left\Vert \Phi
_{D}\right\Vert _{\infty},\nonumber\\
\left\Vert \varphi_{v}+\Phi_{D}\right\Vert _{\infty}  &  \leq\left\Vert
\Phi_{D}\right\Vert _{\infty}. \label{stimaphi2}%
\end{align}

\end{lemma}

\begin{proof}
Fix $v\in H_{0}^{1}\left(  \Omega\right)  $. Let $\tilde{\varphi}_{v}$ be the
unique solution of%
\[
\left\{
\begin{array}
[c]{ll}%
\Delta\varphi=\left(  \varphi-\Phi_{D}^{-}\right)  \left(  v+U\right)
^{2},\quad & \\
\varphi=0 & \text{on }\partial\Omega.
\end{array}
\right.
\]
We claim that
\[
0\leq\tilde{\varphi}_{v}\leq\Phi_{D}^{-},\qquad\text{a.e. in }\Omega.
\]
Indeed, since $\tilde{\varphi}_{v}$ is the minimum of the functional
\[
f\left(  \varphi\right)  =\frac{1}{2}\int_{\Omega}|\nabla\varphi|^{2}%
dx+\frac{1}{2}\int_{\Omega}\varphi^{2}\left(  v+U\right)  ^{2}dx-\int_{\Omega
}\varphi\Phi_{D}^{-}\left(  v+U\right)  ^{2}dx
\]
and
\[
f\left(  |\tilde{\varphi}_{v}|\right)  \leq f\left(  \tilde{\varphi}%
_{v}\right)  ,
\]
we have that $\tilde{\varphi}_{v}$ is positive.

On the other hand, suppose that
\begin{equation}
0<\tilde{\varphi}_{v}-\Phi_{D}^{-},\qquad\text{a.e. in }\Omega_{1}%
\subset\Omega. \label{contr}%
\end{equation}
Since $\tilde{\varphi}_{v}-\Phi_{D}^{-}$ solves%
\[
\left\{
\begin{array}
[c]{ll}%
\Delta w=w\left(  v+U\right)  ^{2} & \text{in }\Omega_{1},\\
w=0 & \text{on }\partial\Omega_{1},\\
w>0 & \text{in }\Omega_{1},
\end{array}
\right.
\]
then%
\[
-\left\Vert \nabla\left(  \tilde{\varphi}_{v}-\Phi_{D}^{-}\right)  \right\Vert
_{2}^{2}=\int_{\Omega_{1}}\left(  \tilde{\varphi}_{v}-\Phi_{D}^{-}\right)
^{2}\left(  v+U\right)  ^{2}dx.
\]
This implies that $\tilde{\varphi}_{v}-\Phi_{D}^{-}=0$ a.e. in $\Omega_{1}$,
which contradicts (\ref{contr}).

Analogously, the unique solution $\hat{\varphi}_{v}$ of%
\[
\left\{
\begin{array}
[c]{ll}%
\Delta\varphi=\left(  \varphi-\Phi_{D}^{+}\right)  \left(  v+U\right)
^{2},\quad & \\
\varphi=0 & \text{on }\partial\Omega,
\end{array}
\right.
\]
satisfies%
\[
0\leq\hat{\varphi}_{v}\leq\Phi_{D}^{+},\qquad\text{a.e. in }\Omega\text{.}%
\]

Thus, linearity and uniqueness imply that%
\[
\tilde{\varphi}_{v}\equiv\varphi_{v}^{+}\qquad\text{and}\qquad\hat{\varphi
}_{v}\equiv\varphi_{v}^{-}%
\]
and estimate (\ref{stimamaxmin}) is proved.
\end{proof}

\begin{proposition}
\label{nice}Under the assumption (\ref{hyp2}), the functional $J$ is bounded
from below, coercive and satisfies the Palais-Smale condition.
\end{proposition}

\begin{proof}
We have%
\begin{align*}
J\left(  v\right)   &  \geq\frac{1}{2}\left\Vert \nabla v\right\Vert _{2}%
^{2}+\frac{m^{2}}{2}\left\Vert v\right\Vert _{2}^{2}-\frac{\left\Vert \Phi
_{D}\right\Vert _{\infty}^{2}}{2}\left\Vert v+U\right\Vert _{2}^{2}\\
&  \geq\frac{1}{2}\left\Vert \nabla v\right\Vert _{2}^{2}-\frac{\left\Vert
\Phi_{D}\right\Vert _{\infty}^{2}-m^{2}}{2}\left\Vert v\right\Vert _{2}%
^{2}-\frac{\left\Vert \Phi_{D}\right\Vert _{\infty}^{2}}{2}\left\Vert
U\right\Vert _{2}^{2}-\left\Vert \Phi_{D}\right\Vert _{\infty}^{2}\left\Vert
v\right\Vert _{2}\left\Vert U\right\Vert _{2}\\
&  \geq\frac{\lambda_{1}-\max\left\{  0,\left\Vert \Phi_{D}\right\Vert
_{\infty}^{2}-m^{2}\right\}  }{2\lambda_{1}}\left\Vert \nabla v\right\Vert
_{2}^{2}-\frac{\left\Vert \Phi_{D}\right\Vert _{\infty}^{2}}{2}\left\Vert
U\right\Vert _{2}^{2}-c\left\Vert \nabla v\right\Vert _{2}.
\end{align*}
Taking into account (\ref{hyp2}), we deduce that $J$ is bounded from below and coercive.

Let $\left\{  v_{n}\right\}  $ be a P-S sequence for $J$, that is a sequence
such that $\left\{  J\left(  u_{n}\right)  \right\}  $ is bounded and
$\left\{  J^{\prime}\left(  v_{n}\right)  \right\}  $ tends to zero. Since $J$
is coercive, the sequence $\left\{  v_{n}\right\}  $ is bounded in $H_{0}%
^{1}\left(  \Omega\right)  $ and, up to subtracting a subsequence, it is
weakly convergent. By (\ref{stimaphi2}), $\left\{  \left(  \varphi_{v_{n}%
}+\Phi_{D}\right)  ^{2}\right\}  $ is uniformly bounded in $L^{\infty}\left(
\Omega\right)  ,$ hence the right hand side of%
\[
\Delta v_{n}=m^{2}v_{n}-\left(  \varphi_{v_{n}}+\Phi_{D}\right)  ^{2}%
(v_{n}+U)-J^{\prime}\left(  v_{n}\right)
\]
is bounded in $H^{-1}\left(  \Omega\right)  $. The claim immediately follows.
\end{proof}

By a standard argument, the functional $J$ has a minimum and the first
statement of Theorem \ref{th1} is thereby proved$.$

Now we prove the second statement.

First we notice that, if $h=0$, then $U=0$. Problem (\ref{prob2}) becomes%
\begin{equation}
\left\{
\begin{array}
[c]{ll}%
-\Delta v-\left(  \varphi+\Phi_{D}\right)  ^{2}v+m^{2}v=0, & \\
\Delta\varphi=\left(  \varphi+\Phi_{D}\right)  v^{2}, & \\
v=\varphi=0 & \text{on }\partial\Omega.
\end{array}
\right.  \label{probchilindir}%
\end{equation}
If $\left(  v,\varphi\right)  $ is a solution of (\ref{probchilindir}), by the
first equation we have
\begin{equation}
\left\Vert \nabla v\right\Vert _{2}^{2}-\int_{\Omega}\left(  \varphi+\Phi
_{D}\right)  ^{2}v^{2}\,dx+m^{2}\left\Vert v\right\Vert _{2}^{2}=0.
\label{banal10}%
\end{equation}
Substituting (\ref{varphiquad}) into (\ref{banal10}) we get%
\[
\left\Vert \nabla v\right\Vert _{2}^{2}+\int_{\Omega}\varphi^{2}%
v^{2}\,dx+2\left\Vert \nabla\varphi\right\Vert _{2}^{2}+\int_{\Omega}\left(
m^{2}-\Phi_{D}^{2}\right)  v^{2}dx=0.
\]
Then%
\begin{align*}
0  &  \geq\lambda_{1}\left\Vert v\right\Vert _{2}^{2}+\int_{\Omega}\varphi
^{2}v^{2}\,dx+2\left\Vert \nabla\varphi\right\Vert _{2}^{2}+\int_{\Omega
}\left(  m^{2}-\Phi_{D}^{2}\right)  v^{2}dx\\
&  \geq\left(  \lambda_{1}+m^{2}-\left\Vert \Phi_{D}\right\Vert _{\infty}%
^{2}\right)  \left\Vert v\right\Vert _{2}^{2}+2\left\Vert \nabla
\varphi\right\Vert _{2}^{2}%
\end{align*}
and, by (\ref{hyp2}), we conclude that $v=\varphi=0$. Since $u=v$, the claim follows.

\section{Proof of Theorem \ref{thmix}}

We consider again the function $U\neq0$ defined by (\ref{def U}).

On the other hand, let $\Phi_{N}$ denote the unique solution of
\[
\left\{
\begin{array}
[c]{ll}%
\Delta\Phi_{N}=q\kappa, & \\
{\dfrac{\partial\Phi_{N}}{\partial\mathbf{n}}}=q\theta & \text{on }%
\partial\Omega,\vspace{3pt}\\
{\int_{\Omega}\Phi_{N}\,dx=0.\quad} &
\end{array}
\right.
\]
with
\[
\kappa=\frac{1}{\left\vert \Omega\right\vert }\int_{\partial\Omega}%
\theta\,d\sigma.
\]
It is well known that $\left\Vert \Phi_{N}\right\Vert _{\infty}\leq
c\left\vert q\right\vert \left\Vert \theta\right\Vert _{H^{1/2}\left(
\partial\Omega\right)  }$: we choose $\left\vert q\right\vert \!\left\Vert
\theta\right\Vert _{H^{1/2}\left(  \partial\Omega\right)  }$ small enough to
get
\[
m^{2}-\Phi_{N}^{2}\geq0.
\]

If we set%
\[
v=u-U,\qquad\varphi=\phi-\Phi_{N},
\]
then, problem (\ref{Main1})-(\ref{Mix1}) becomes%
\begin{equation}
\left\{
\begin{array}
[c]{ll}%
-\Delta v-\left(  \varphi+\Phi_{N}\right)  ^{2}\left(  v+U\right)
+m^{2}v=0,\quad & \\
\Delta\varphi=\left(  \varphi+\Phi_{N}\right)  \left(  v+U\right)
^{2}-q\kappa, & \\
v=0 & \text{on }\partial\Omega,\\
\displaystyle{\frac{\partial\varphi}{\partial\mathbf{n}}}=0 & \text{on
}\partial\Omega.
\end{array}
\right.  \label{mixprobl2}%
\end{equation}

The following lemma has been proved in \cite[Lemma 2.3]{DPS2007}.

\begin{lemma}
For every $w\in H^{1}\left(  \Omega\right)  \setminus\left\{  0\right\}  $ and
$\rho\in L^{6/5}\left(  \Omega\right)  $, the problem%
\[
\left\{
\begin{array}
[c]{ll}%
-\Delta\varphi+\varphi w^{2}=\rho,\quad & \\
{\dfrac{\partial\varphi}{\partial\mathbf{n}}=0} & \text{on }\partial\Omega
\end{array}
\right.
\]
has a unique solution in $H^{1}\left(  \Omega\right)  $.
\end{lemma}

Since $U=\sqrt{4\pi}qh\neq0$ on $\partial\Omega$, for every $v\in H_{0}%
^{1}\left(  \Omega\right)  $ we have $v+U\neq0$. Hence, by the previous lemma,
the problem
\[
\left\{
\begin{array}
[c]{ll}%
\Delta\varphi=\left(  \varphi+\Phi_{N}\right)  \left(  v+U\right)
^{2}-q\kappa, & \\
\displaystyle{\frac{\partial\varphi}{\partial\mathbf{n}}}=0 & \text{on
}\partial\Omega
\end{array}
\right.
\]
has always a unique solution $\varphi_{v}\in H^{1}\left(  \Omega\right)  $.

As well as in the previous section, we use a variational principle: we look
for the critical points of a \emph{reduced} functional $J=J(v)$ defined in
$H_{0}^{1}\left(  \Omega\right)  $; then the solutions of (\ref{mixprobl2})
are the pairs $(v,\varphi_{v})$.

The reduced functional has the form%

\begin{multline*}
J\left(  v\right)  =\frac{1}{2}\left\Vert \nabla v\right\Vert _{2}^{2}%
+\frac{1}{2}\int_{\Omega}\left(  m^{2}-\Phi_{N}^{2}\right)  \left(
v+U\right)  ^{2}dx+\\
-\frac{1}{2}\int_{\Omega}\Phi_{N}\left(  v+U\right)  ^{2}\varphi_{v}%
\,dx+\frac{q\kappa}{2}\int_{\Omega}\varphi_{v}\,dx.
\end{multline*}

With the aim of studying the functional $J$, as in \cite[Section 3]{DPS2007},
we consider
\[
\varphi_{v}=\xi_{v}+\eta_{v},
\]
where $\xi_{v}$ and $\eta_{v}$ solve respectively%
\[
\left\{
\begin{array}
[c]{ll}%
\Delta\xi-\left(  v+U\right)  ^{2}\xi=\left(  v+U\right)  ^{2}\Phi_{N},\quad &
\\
\frac{\partial\xi}{\partial\mathbf{n}}=0 & \text{on }\partial\Omega,
\end{array}
\right.
\]%
\[
\left\{
\begin{array}
[c]{ll}%
\Delta\eta-\left(  v+U\right)  ^{2}\eta=-q\kappa,\quad & \\
\frac{\partial\eta}{\partial\mathbf{n}}=0 & \text{on }\partial\Omega.
\end{array}
\right.
\]
We obtain the following estimates%
\begin{gather*}
\int_{\Omega}\xi_{v}\Phi_{N}\left(  v+U\right)  ^{2}dx\leq0,\\
-\max\Phi_{N}\leq\xi_{v}\leq-\min\Phi_{N},\\
q\kappa\eta_{v}\geq0,\\
\left\Vert \nabla\eta_{v}\right\Vert _{2}\leq c\left\vert \bar{\eta}%
_{v}\right\vert \left\Vert v+U\right\Vert _{4}^{2},
\end{gather*}
where $c>0$ and $\bar{\eta}_{v}$ is the average of $\eta_{v}$.

Then we can show that the functional $J$ is coercive, bounded from below and
satisfies the PS condition. So we get the existence of a minimum of $J$, which
gives rise to a solution of (\ref{mixprobl2}).

\section{Proof of Theorem \ref{thnonlin}}

With the same change of variables of Section \ref{DDsect} (with $U=0$, since
$h=0$), problem (\ref{prob1nl})-(\ref{Dir2}) can be written as%
\begin{equation}
\left\{
\begin{array}
[c]{ll}%
-\Delta v-\left(  \varphi+\Phi_{D}\right)  ^{2}v+m^{2}v-g\left(  x,v\right)
=0,\quad & \\
\Delta\varphi=\left(  \varphi+\Phi_{D}\right)  v^{2}, & \\
v=\varphi=0 & \text{on }\partial\Omega.
\end{array}
\right.  \label{probchidir}%
\end{equation}
As in the previous sections, the solutions of (\ref{probchidir}) have the form
$(v,\varphi_{v})$, where $v$ is a critical point of the \emph{reduced}
functional%
\[
J\left(  v\right)  ={\frac{1}{2}\left\Vert \nabla v\right\Vert _{2}^{2}%
+\frac{m^{2}}{2}\int_{\Omega}v^{2}dx-\frac{1}{2}\int_{\Omega}\Phi_{D}\left(
\varphi_{v}+\Phi_{D}\right)  v^{2}\,dx-\int_{\Omega}G\left(  x,v\right)  \,dx}%
\]
which is $C^{1}$ in $H_{0}^{1}\left(  \Omega\right)  $.

Let us recall that the well-known conditions $\mathbf{\left(  g_{1}\right)
}-\mathbf{\left(  g_{3}\right)  }$ imply that:

\begin{enumerate}
\item[$\left(  \mathbf{G}_{1}\right)  $] for every $\varepsilon>0$ there
exists $A\geq0$ such that for every $t\in\mathbf{R}$
\[
\left|  G\left(  x,t\right)  \right|  \leq\frac{\varepsilon}{2} t^{2}+A\left|
t\right|  ^{p};
\]

\item[$\left(  \mathbf{G}_{2}\right)  $] there exist two constants
$b_{1},b_{2}>0$ such that for every $t\in\mathbf{R}$
\[
G\left(  x,t\right)  \geq b_{1}\left\vert t\right\vert ^{s}-b_{2}.
\]

\end{enumerate}

Now we can state some properties of $J$.

\begin{lemma}
\label{PS lemma}The functional $J$ satisfies the Palais-Smale condition on
$H_{0}^{1}(\Omega)$ and diverges negatively on every finite dimensional
subspace of $H_{0}^{1}(\Omega)$.
\end{lemma}

\begin{proof}
Let $\left\{  v_{n}\right\}  \subset H_{0}^{1}\left(  \Omega\right)  $ such
that
\begin{align}
&  \left\vert J\left(  v_{n}\right)  \right\vert \leq c_{1}\label{PS1}\\
&  J^{\prime}\left(  v_{n}\right)  \rightarrow0. \label{PS2}%
\end{align}
As before, we set $\varphi_{n}=\varphi_{v_{n}}$ and we use $c_{i}$ to denote
suitable positive constants. By (\ref{PS1}), $\left(  \mathbf{g}_{1}\right)  $
and $\left(  \mathbf{g}_{3}\right)  $%
\begin{align}
\frac{1}{2}\left\Vert \nabla v_{n}\right\Vert _{2}^{2}  &  \leq c_{1}%
+\int_{\Omega}G\left(  x,v_{n}\right)  \,dx+\frac{1}{2}\int_{\Omega}\Phi
_{D}\left(  \varphi_{n}+\Phi_{D}\right)  v_{n}^{2}\,dx\nonumber\\
&  \leq c_{2}+\frac{1}{s}\int_{\left\{  x\in\Omega:\left\vert v_{n}\left(
x\right)  \right\vert \geq r\right\}  }g\left(  x,v_{n}\right)  v_{n}%
\,dx+\frac{1}{2}\left\Vert \Phi_{D}\right\Vert _{\infty}^{2}\left\Vert
v_{n}\right\Vert _{2}^{2}\nonumber\\
&  \leq c_{3}+\frac{1}{s}\int_{\Omega}g\left(  x,v_{n}\right)  v_{n}%
\,dx+\frac{1}{2}\left\Vert \Phi_{D}\right\Vert _{\infty}^{2}\left\Vert
v_{n}\right\Vert _{2}^{2}. \label{1relaz}%
\end{align}

On the other hand, by (\ref{PS2}),
\begin{align*}
\left\vert \left\Vert \nabla v_{n}\right\Vert _{2}^{2}+m^{2}\left\Vert
v_{n}\right\Vert _{2}^{2}-\int_{\Omega}\left(  \varphi_{n}+\Phi_{D}\right)
^{2}v_{n}^{2}\,dx-\int_{\Omega}g\left(  x,v_{n}\right)  v_{n}\,dx\right\vert
&  =\left\vert \left\langle J^{\prime}\left(  v_{n}\right)  ,v_{n}%
\right\rangle \right\vert \\
&  \leq c_{4}\left\Vert \nabla v_{n}\right\Vert _{2},
\end{align*}
hence
\begin{align}
\int_{\Omega}g\left(  x,v_{n}\right)  v_{n}\,dx  &  \leq c_{4}\left\Vert
\nabla v_{n}\right\Vert _{2}+\left\Vert \nabla v_{n}\right\Vert _{2}^{2}%
+m^{2}\left\Vert v_{n}\right\Vert _{2}^{2}-\int_{\Omega}\left(  \varphi
_{n}+\Phi_{D}\right)  ^{2}v_{n}^{2}\,dx\nonumber\\
&  \leq c_{4}\left\Vert \nabla v_{n}\right\Vert _{2}+\left\Vert \nabla
v_{n}\right\Vert _{2}^{2}+m^{2}\left\Vert v_{n}\right\Vert _{2}^{2}.
\label{2relaz}%
\end{align}
Therefore, substituting (\ref{2relaz}) into (\ref{1relaz}), we easily find
\[
\frac{s-2}{2s}\left\Vert \nabla v_{n}\right\Vert _{2}^{2}\leq c_{3}%
+\frac{c_{4}}{s}\left\Vert \nabla v_{n}\right\Vert _{2}+c_{5}\left\Vert
v_{n}\right\Vert _{2}^{2}%
\]
or, equivalently,%
\begin{equation}
\left\Vert v_{n}\right\Vert _{2}^{2}\geq c_{6}\left\Vert \nabla v_{n}%
\right\Vert _{2}^{2}-c_{7}\left\Vert \nabla v_{n}\right\Vert _{2}-c_{8}.
\label{priori}%
\end{equation}

Now we claim that $\left\{  v_{n}\right\}  $ is bounded in $H_{0}^{1}\left(
\Omega\right)  $. Otherwise, by (\ref{priori}), up to a subsequence, we have
\[
\left\Vert v_{n}\right\Vert _{2}^{2}\geq c_{9}\left\Vert \nabla v_{n}%
\right\Vert _{2}^{2}\rightarrow+\infty
\]
and then, by $\left(  \mathbf{G}_{2}\right)  $,%
\begin{align*}
J(v_{n})  &  \leq\frac{1}{2}\left\Vert \nabla v_{n}\right\Vert _{2}^{2}%
+\frac{m^{2}}{2}\Vert v_{n}\Vert_{2}^{2}+\frac{\Vert\Phi_{D}\Vert_{\infty}%
^{2}}{2}\Vert v_{n}\Vert_{2}^{2}-\int_{\Omega}G\left(  x,v_{n}\right)  dx\\
&  \leq c_{10}\Vert v_{n}\Vert_{2}^{2}-b_{1}\Vert v_{n}\Vert_{s}^{s}%
+b_{2}|\Omega|\rightarrow-\infty,
\end{align*}
which contradicts (\ref{PS1}).

Thus, we can assume that
\[
v_{n}\rightharpoonup v\text{ in }H_{0}^{1}\left(  \Omega\right)  .
\]
This convergence is strong. Indeed, by $\left(  \mathbf{g}_{1}\right)  $ and
the same arguments used in the proof of Proposition \ref{nice}, we get that
the right hand side of%
\[
\Delta v_{n}=m^{2}v_{n}-\left(  \varphi_{n}+\Phi_{D}\right)  ^{2}%
v_{n}-g\left(  x,v_{n}\right)  -J^{\prime}\left(  v_{n}\right)
\]
is bounded in $H^{-1}\left(  \Omega\right)  $.

Finally, by $\left(  \mathbf{G}_{2}\right)  $,
\begin{align*}
J\left(  v\right)   &  \leq\frac{1}{2}\left\Vert \nabla v\right\Vert _{2}%
^{2}+c_{1}\left\Vert v\right\Vert _{2}^{2}-\int_{\Omega}G(x,v)\,dx\\
&  \leq\frac{1}{2}\left\Vert \nabla v\right\Vert _{2}^{2}+c_{1}\left\Vert
v\right\Vert _{2}^{2}-b_{1}\left\Vert v\right\Vert _{s}^{s}+b_{2}\left\vert
\Omega\right\vert \rightarrow-\infty
\end{align*}
when $\left\Vert v\right\Vert \rightarrow+\infty$ on every finite dimensional subspace.
\end{proof}

The first part of Theorem \ref{thnonlin} follows from the classical Mountain
Pass Theorem. Indeed from $\left(  \mathbf{G}_{1}\right)  $ and Lemma
\ref{stimaphi} we deduce
\begin{align}
J\left(  v\right)   &  \geq\frac{1}{2}\left\Vert \nabla v\right\Vert _{2}%
^{2}+\frac{m^{2}}{2}\left\Vert v\right\Vert _{2}^{2}-\frac{\left\Vert \Phi
_{D}\right\Vert _{\infty}^{2}}{2}\left\Vert v\right\Vert _{2}^{2}%
-\frac{\varepsilon}{2}\left\Vert v\right\Vert _{2}^{2}-A\left\Vert
v\right\Vert _{p}^{p}\nonumber\\
&  \geq\frac{\lambda_{1}-\max\left\{  0,\left\Vert \Phi_{D}\right\Vert
_{\infty}^{2}-m^{2}\right\}  -\varepsilon}{2\lambda_{1}}\left\Vert \nabla
v\right\Vert _{2}^{2}-A^{\prime}\left\Vert \nabla v\right\Vert _{2}^{p},
\label{stimaJ1}%
\end{align}
with $A,A^{\prime}>0$ depending on $\varepsilon$. Taking into account
(\ref{hyp2}), if we choose $\varepsilon$ sufficiently small, we have
\begin{equation}
J\left(  v\right)  \geq c\left\Vert \nabla v\right\Vert _{2}^{2}-A^{\prime
}\left\Vert \nabla v\right\Vert _{2}^{p} \label{stimaJ2}%
\end{equation}
with $c>0$. Hence $J$ has a strict local minimum in $0$. By Lemma
\ref{PS lemma}, the claim immediately follows.

\begin{remark}
The proof of Lemma \ref{PS lemma} does not use any assumption on the value
$\omega^{2}-m^{2}$. Hence, if $\omega^{2}-m^{2}\neq\lambda_{k}$ (and $q$ is
sufficiently small) we conjecture that a critical point could be obtained by
means of some variant of the Mountain Pass Theorem.
\end{remark}

If $g$ is odd, the functional $J$ is even and we apply the following
$\mathbf{Z_{2}}$-Mountain Pass Theorem (see \cite{rab}).

\begin{theorem}
\label{Z2MP} Let $E$ be an infinite dimensional Banach space and let $I\in
C^{1}\left(  E,\mathbf{R}\right)  $ be even, satisfy the Palais-Smale
condition and $I\left(  0\right)  =0.$ If $E=V\oplus X,$ where $V$ is finite
dimensional and $I$ satisfies

\begin{enumerate}
\item there are constants $\rho,\alpha>0$ such that $\left.  I\right\vert
_{\partial B_{\rho}\cap X}\geq\alpha,$ and

\item for each finite dimensional subspace $\tilde{E}\subset E,$ there is an
$R=R(\tilde{E})$ such that $I\leq0$ on $E\setminus B_{R(\tilde{E})}$,
\end{enumerate}

\noindent then $I$ possesses an unbounded sequence of critical values.
\end{theorem}

We have only to prove the first geometrical condition. We distinguish two cases.

\begin{enumerate}
\item[(a)] If $\left\Vert \Phi_{D}\right\Vert _{\infty}^{2}-m^{2}<\lambda_{1}%
$, we proceed as in (\ref{stimaJ1}) and (\ref{stimaJ2}) and we obtain that $J$
has a strict local minimum in $0$. Applying Theorem \ref{Z2MP} with
$V=\left\{  0\right\}  $ we infer the existence of infinitely many solutions
{$v$}${_{i}}$.

\item[(b)] If $\lambda_{1}\leq\left\Vert \Phi_{D}\right\Vert _{\infty}%
^{2}-m^{2}$, denoted with $\left\{  \lambda_{j}\right\}  $ the eigenvalues of
$-\Delta$ with Dirichlet boundary condition, $M_{j}$ the (finite dimensional)
corresponding eigenspaces and
\[
k=\min\left\{  j\in\mathbf{N}:\left\Vert \Phi_{D}\right\Vert _{\infty}%
^{2}-m^{2}<\lambda_{j}\right\}  ,
\]
we consider
\[
V=\bigoplus_{j=1}^{k-1}M_{j},\qquad X=V^{\perp}=\overline{\bigoplus
_{j=k}^{+\infty}M_{j}}.
\]

Since
\[
\lambda_{k}=\min\left\{  \frac{\left\Vert \nabla v\right\Vert _{2}^{2}%
}{\left\Vert v\right\Vert _{2}^{2}}:v\in X,\;v\neq0\right\}  ,
\]
for every $v\in X$ we have that
\[
J\left(  v\right)  \geq\frac{\lambda_{k}-\left(  \left\Vert \Phi
_{D}\right\Vert _{\infty}^{2}-m^{2}\right)  }{2\lambda_{k}}\left\Vert \nabla
v\right\Vert _{2}^{2}-\int_{\Omega}G(x,v)\,dx.
\]
Thus, arguing as before, $J$ is strictly positive on a sphere in $X$ and we
obtain the existence of infinitely many finite energy solutions $v_{i}$.
\end{enumerate}

In both cases we have $J(v_{i})\rightarrow+\infty$. To complete the proof of
Theorem \ref{thnonlin}, we simply notice that, by Lemma \ref{stimaphi} and
$\left(  \mathbf{G}_{1}\right)  $,%
\begin{align*}
J\left(  v_{i}\right)   &  =\frac{1}{2}\left\Vert \nabla v_{i}\right\Vert
_{2}^{2}+\frac{m^{2}}{2}\left\Vert v_{i}\right\Vert _{2}^{2}-\frac{1}{2}%
\int_{\Omega}\Phi_{D}\left(  \varphi_{i}+\Phi_{D}\right)  v_{i}^{2}%
\,dx-\int_{\Omega}G\left(  x,v_{i}\right)  \,dx\\
&  \leq c_{1}\Vert\nabla v_{i}\Vert_{2}^{2}+c_{3}\left\Vert \nabla
v_{i}\right\Vert _{2}^{p}.
\end{align*}

\end{document}